\documentclass[12pt]{article}
\usepackage{cite}
\usepackage{mathrsfs}
\usepackage{amsmath,amsthm,amssymb,amsfonts}
\usepackage{a4,latexsym,epsfig}
\usepackage{color,colordvi,amscd}

\newtheorem{theorem}{Theorem}[section]

\newtheorem{lemma}[theorem]{Lemma}
\newtheorem{corollary}[theorem]{Corollary}

\title{Connectivity, toughness, spanning trees of bounded degree, and the spectrum of regular graphs}
\author{Sebastian M. Cioab\u{a}\thanks{
Department of Mathematical Sciences,
University of Delaware, Newark, DE 19716, USA. E-mail: {\tt cioaba@udel.edu}. This work was partially supported by the National Security Agency grant H98230-13-0267.} \,  and 
Xiaofeng Gu\thanks{
Department of Mathematics, University of West Georgia, Carrollton, GA 30118, USA. E-mail: {\tt xgu@westga.edu}}
}

\begin{document}
\date{\today}
\maketitle

\begin{center}
{\em Dedicated to the memory of Professor Miroslav Fiedler}
\end{center}

\noindent
\begin{abstract}
In this paper, we present some new results describing connections between the spectrum of a regular graph and its generalized connectivity, toughness, and the existence of spanning trees with bounded degree.
\end{abstract}


\section{Introduction}

The spectrum of a graph is related to many important combinatorial parameters. In his fundamental and ground-breaking work, Fiedler \cite{F73,F75} determined close connections between the Laplacian eigenvalues and eigenvectors of a graph and combinatorial parameters such as its vertex-connectivity or edge-connectivity. Fiedler's work stimulated tremendous progress and growth in spectral graph theory since then.

In this paper, we study the connections between the spectrum of a regular graph and other combinatorial parameters such as generalized connectivity, toughness and the existence of spanning trees with bounded degree. 

Throughout this paper, we consider only finite, undirected and simple graphs. Given a graph $G=(V,E)$ of order $n$, we denote by $\lambda_1(G) \geq \lambda_2(G)\geq \dots \geq \lambda_n(G)$ the eigenvalues of its adjacency matrix. When the graph $G$ is clear from the context, we use $\lambda_i$ to denote $\lambda_i(G)$. We also use the notation $\lambda=\max\{|\lambda_2|, |\lambda_n|\}$. If $G$ is $d$-regular, then $\lambda_1=d$ and the multiplicity of $d$ equals the number of components of $G$.   We use $\kappa(G), \kappa'(G)$ and $c(G)$ to denote the vertex-connectivity, the edge-connectivity and the number of components of a graph $G$, respectively. For any undefined graph theoretic notions, see Bondy and Murty \cite{BoMu08} or Brouwer and Haemers \cite{BrHa12}.

One of well-known results of Fiedler \cite{F73} implies that the vertex-connectivity of a $d$-regular graph is at least $d-\lambda_2$. This result was improved in certain ranges by Krivelevich and Sudakov \cite{KS06} who showed that the vertex-connectivity of a $d$-regular graph is at least $d-\frac{36\lambda^2}{d}$. Given an integer $l\ge 2$, Chartrand, Kapoor, Lesniak and Lick \cite{CKLL84} defined the  {\bf $l$-connectivity $\kappa_l(G)$} of a graph $G$ to be the minimum number of vertices of $G$ whose removal produces a disconnected graph with at least $l$ components or a graph with fewer than $l$ vertices. Thus $\kappa_l(G)=0$ 
if and only if $c(G)\ge l$ or $|V(G)|\le l-1$. Note that $\kappa_2 (G)=\kappa (G)$.
For $k\ge 1$, a graph $G$ is called {\bf $(k,l)$-connected} if $\kappa_l\ge k$.
See \cite{CKLL84,DaOS99,GuLY11,Oell87} for more about $l$-connectivity and $(k,l)$-connected graphs.
In particular, a structural characterization of  $(2,l)$-connected graphs is presented in \cite{GuLY11}, as a generalization of the standard characterization of $2$-connected graphs (see \cite[Chapter 5]{BoMu08}).

Our results relating generalized connectivity to the spectrum of a regular graph are below.

\begin{theorem}\label{L-conn}
Let $l,k$ be integers with $l\ge k\ge 2$. For any connected $d$-regular graph $G$ with 
$|V(G)|\ge k+l-1$, $d\ge 3$ and edge connectivity $\kappa'$, if $\kappa'=d$, or,
if $\kappa' < d$ and
\[
\lambda_{\lceil\frac{(l-k+1)d}{d-\kappa'}\rceil} (G)
< \left\{ \begin{array}{ll}
        \frac{d-2 + \sqrt{d^2 +12}}{2}, & \textrm{if $d$ is even},\\
        \frac{d-2 + \sqrt{d^2 +8}}{2}, & \textrm{if $d$ is odd},
        \end{array} \right.
\]
then $\kappa_l(G)\ge k$.
\end{theorem}

\begin{corollary}\label{L-conn2}
Let $l\ge 2$. For any connected $d$-regular graph $G$ with $|V(G)|\ge l+1$ and $d\ge 3$, if 
\[
\lambda_l (G)
< \left\{ \begin{array}{ll}
        \frac{d-2 + \sqrt{d^2 +12}}{2}, & \textrm{if $d$ is even},\\
        \frac{d-2 + \sqrt{d^2 +8}}{2}, & \textrm{if $d$ is odd},
        \end{array} \right.
\]
then $\kappa_l(G)\ge 2$.
\end{corollary}

\begin{corollary}\label{2-conn}
For any connected $d$-regular graph $G$ with $d\ge 3$, if 
\[
\lambda_2 (G)
< \left\{ \begin{array}{ll}
        \frac{d-2 + \sqrt{d^2 +12}}{2}, & \textrm{if $d$ is even},\\
        \frac{d-2 + \sqrt{d^2 +8}}{2}, & \textrm{if $d$ is odd},
        \end{array} \right.
\]
then $\kappa(G)\ge 2$.
\end{corollary}
Corollary~\ref{2-conn} is a slight improvement of previous results of 
Krivelevich and Sudakov \cite[Theorem 4.1]{KS06} and Fiedler \cite[Theorem 4.1]{F73}.

The toughness $t(G)$ of a connected graph $G$ is defined as $t(G)=\min\{\frac{|S|}{c(G-S)}\}$,
where the minimum is taken over all proper subset $S\subset V(G)$ such that $c(G-S)>1$.
A graph $G$ is {\bf $t$-tough} if $t(G)\ge t$. This parameter was introduced by Chv\'atal \cite{Chva73} in 1973 and is closely related to many graph properties, including Hamiltonicity, pancyclicity
and spanning trees, see \cite{BaBS06}.  By definitions of toughness and generalized connectivity, for a noncomplete connected graph $G$, we have $t(G)=\min_{2\le l\le \alpha} \{\frac{\kappa_l(G)}{l}\}$ where $\alpha$ is the independence number of $G$ (see also \cite{DaOS99}). 

The relationship between the toughness of a regular graph and eigenvalues has been considered by many researchers, among which Alon \cite{Alon95} is the first one. 
\begin{theorem}[Alon \cite{Alon95}]
For any connected $d$-regular graph $G$, $t(G)>\frac{1}{3}(\frac{d^2}{d\lambda+\lambda^2}-1)$.
\end{theorem}
Around the same time, Brouwer \cite{Brou95} independently discovered a slightly better bound of $t(G)$. 
\begin{theorem}[Brouwer \cite{Brou95}]
For any connected $d$-regular graph $G$, $t(G)>\frac{d}{\lambda}-2$.
\end{theorem}
Brouwer~\cite{Brou96} conjectured that the lower bound of the previous theorem can be improved
to $\frac{d}{\lambda}-1$ for any connected $d$-regular graph $G$. For the special case of toughness $1$, Liu and Chen \cite{LiCh10} improved Brouwer's previous result.

\begin{theorem}[Liu and Chen \cite{LiCh10}]
For any connected $d$-regular graph $G$, if 
\[
\lambda_2(G)
< \left\{ \begin{array}{ll}
        d -1 + \frac{3}{d+1}, & \textrm{if $d$ is even},\\
        d -1 +\frac{2}{d+1}, & \textrm{if $d$ is odd},
        \end{array} \right.
\]
then $t(G)\ge 1$.
\end{theorem}

Recently, Cioab\u{a} and Wong \cite{CiWo14} further improved the above result.
\begin{theorem}[Cioab\u{a} and Wong \cite{CiWo14}]
\label{cwthm15}
For any connected $d$-regular graph $G$, if 
\[
\lambda_2(G)
< \left\{ \begin{array}{ll}
        \frac{d-2 + \sqrt{d^2 +12}}{2}, & \textrm{if $d$ is even},\\
        \frac{d-2 + \sqrt{d^2 +8}}{2}, & \textrm{if $d$ is odd},
        \end{array} \right.
\]
then $t(G)\ge 1$.
\end{theorem}
Moreover, Cioab\u{a} and Wong \cite{CiWo14} showed the previous result is best possible by constructing $d$-regular graphs whose second largest eigenvalues equals the right hand-side of the previous theorem,  but having toughness less than $1$. An immediate corollary of the previous result is the following.

\begin{corollary}[Cioab\u{a} and Wong \cite{CiWo14}]
\label{cwcor23}
For any bipartite connected $d$-regular graph $G$, if 
\[
\lambda_2(G)
< \left\{ \begin{array}{ll}
        \frac{d-2 + \sqrt{d^2 +12}}{2}, & \textrm{if $d$ is even},\\
        \frac{d-2 + \sqrt{d^2 +8}}{2}, & \textrm{if $d$ is odd},
        \end{array} \right.
\]
then $t(G)=1$.
\end{corollary}

These authors also found the second largest eigenvalue condition for $t(G)\ge \tau$, where $\tau\le \kappa'/d$ is a positive number. 

\begin{theorem}[Cioab\u{a} and Wong \cite{CiWo14}]
\label{cwthm14}
Let $G$ be a connected $d$-regular graph with edge connectivity $\kappa'$ and $d\ge 3$.
Suppose that $\tau$ is a positive number with $\tau\le \kappa'/d$. 
If $\lambda_2(G)< d-\frac{\tau d}{d+1}$, then $t(G)\ge \tau$.
\end{theorem}

In this paper, we continue to investigate the relationship between toughness of a regular
graph and its eigenvalues. The following theorems are the main results.
As $\lceil\frac{d}{d-\kappa'}\rceil\ge 2$, Theorem \ref{onetough} is an improvement of 
Theorem \ref{cwthm15}. For bipartite regular graphs, Theorem~\ref{bipar} improves 
Corollary~\ref{cwcor23}. We shall also mention that in Theorem \ref{cwthm14} the eigenvalue
condition is not needed, see Theorem \ref{tautough}. As an application of Theorem \ref{tautough}, 
Corollary \ref{brouconj} confirms a conjecture of Brouwer \cite{Brou96} when $\kappa'<d$.

\begin{theorem}\label{onetough}
Let $G$ be a connected $d$-regular graph with $d\ge 3$ and edge connectivity $\kappa'$.
If $\kappa' =d$, or, if $\kappa'< d$ and
\[
\lambda_{\lceil\frac{d}{d-\kappa'}\rceil} (G)
< \left\{ \begin{array}{ll}
        \frac{d-2 + \sqrt{d^2 +12}}{2}, & \textrm{if $d$ is even},\\
        \frac{d-2 + \sqrt{d^2 +8}}{2}, & \textrm{if $d$ is odd},
        \end{array} \right.
\]
then $t(G)\ge 1$.
\end{theorem}

\begin{theorem}
\label{bipar}
For any bipartite connected $d$-regular graph $G$ with $\kappa' < d$, if 
$\lambda_{\lceil\frac{d}{d-\kappa'}\rceil} (G) < d-\frac{d-1}{2d}$, then $t(G)=1$.
\end{theorem}

\begin{theorem}
\label{tautough}
Let $G$ be a connected $d$-regular graph with edge connectivity $\kappa'$.
Then $t(G)\ge \kappa'/d$.
\end{theorem}

\begin{corollary}
\label{brouconj}
For any connected $d$-regular graph $G$ with $d\ge 3$ and  edge connectivity $\kappa'<d$, 
$t(G)>\frac{d}{\lambda_2}-1\ge \frac{d}{\lambda}-1$.
\end{corollary}

Recently, there has been a lot of activity concerning connections between eigenvalues of a graph and the maximum number of edge-disjoint spanning trees that can be packed in the graph \cite{CiWo12, Gu13,Gu16,GLLY12, LiSh13, LiHL13, LHGL14, Wong13}. Another interesting problem would be to see how eigenvalues of a graph influence the types of spanning trees contained in it. For an integer $k\ge 2$, a {\bf $k$-tree} is a tree with the maximum degree at most $k$. This topic is related to connected factors. A {\bf $[1,k]$-factor} is a spanning subgraph in which each vertex has the degree at least one and at most $k$. By definition, a graph $G$ has a spanning $k$-tree if and only if $G$ has a connected $[1,k]$-factor. For more about degree bounded trees, we refer readers to the survey \cite{OzYa11}. For spectral conditions of $k$-factors in regular graphs, please see \cite{CGH09,Gu14+,Lu10, Lu12}.  In his PhD Dissertation, Wong \cite{Wong13} proved the following sufficient spectral condition for the existence of spanning $k$-trees in regular graphs for $k\ge 3$.

\begin{theorem}[Wong~\cite{Wong13}]
Let $k\ge 3$ and $G$ be a connected $d$-regular graph. If $\lambda_4 < d -\frac{d}{(k-2)(d+1)}$,
then $G$ has a spanning $k$-tree.
\end{theorem}

In this paper, we improve this result.
\begin{theorem}\label{kspanning}
Let $k\ge 3$ and $G$ be a connected $d$-regular graph with edge connectivity $\kappa'$.
Let $l=d-(k-2)\kappa'$. Each of the following statements holds.
\\(i) If $l\le 0$, then $G$ has a spanning $k$-tree.
\\(ii) If $l > 0$ and $\lambda_{\lceil\frac{3d}{l}\rceil} < d -\frac{d}{(k-2)(d+1)}$,
then $G$ has a spanning $k$-tree.
\end{theorem}

Note that eigenvalue conditions for the existence of spanning $2$-trees (Hamiltonian paths) and Hamiltonian cycles have been obtained by Krivelevich and Sudakov \cite{KS03} and Butler and Chung \cite{BC10}.

\section{Preliminaries}

In this section, we present some eigenvalue interlacing results to be used in our arguments. For a real and symmetric matrix $M$ of order $n$ and a natural number $1\leq i\leq n$, we denote by $\lambda_i(M)$ the $i$-th largest eigenvalue of $M$. The following interlacing theorem can be found in many textbooks, for example, \cite[page 35]{BrHa12} or \cite[page 193]{GoRo01}, and is usually referred to as Cauchy eigenvalue interlacing.
\begin{theorem}
Let $A$ be a real symmetric $n\times n$ matrix and $B$ be a principal $m\times m$ submatrix of $A$.
Then $\lambda_i(A)\ge \lambda_i(B)\ge \lambda_{n-m+i}(A)$ for $1\le i\le m$.
\end{theorem}

\begin{corollary}
\label{lamp}
Let $S_1, S_2,\cdots, S_p$ be disjoint subsets of $V(G)$ with $e(S_i,S_j)=0$ for $i\neq j$. For $1\leq i\leq p$, let $G[S_i]$ denote the subgraph of $G$ induced by $S_i$. Then
\[
\lambda_p(G)\ge \lambda_p(G[\cup^p_{i=1} S_i])\ge\min_{1\le i\le p} \{\lambda_1(G[S_i])\}.
\]
\end{corollary}

Let $d\ge 3$ be an integer, and $\mathcal{X}(d)$ denote the family of all connected irregular graphs 
with maximum degree $d$, order $n\ge d+1$ and size $m$ with $2m\ge dn-d+1$ that have at least
two vertices of degree $d$ if $d$ is odd and at least three vertices of degree $d$ if $d$ is even. If $t\geq 2$ is an even integer, let $M_t$ denote the disjoint union of $t/2$ edges. If $G$ and $H$ are two vertex disjoint graphs, the join $G\vee H$ of $G$ and $H$ is the graph obtained by taking the union of $G$ and $H$ and adding all the edges between the vertex set of $G$ and the vertex set of $H$. The complement of $G$ is denoted by $\overline{G}$. For $d\geq 3$, define $X_d$ as $\overline{M}_{d-1}\vee K_2$ if $d$ is odd and $\overline{M}_{d-2}\vee K_3$ if $d$ is even.
\begin{lemma}[Cioab\u{a} and Wong \cite{CiWo14}]
\label{cwlem}
Let $d\ge 3$ be an integer and $H\in\mathcal{X}(d)$. Then
\[
\lambda_1(H)\ge \theta(d)=\left\{ \begin{array}{ll}
        \frac{1}{2}(d -2 +\sqrt{d^2 +12}), & \textrm{if $d$ is even},\\
        \frac{1}{2}(d -2 +\sqrt{d^2 +8} ),& \textrm{if $d$ is odd}.
        \end{array} \right.
\]
Equality happens if and only if $G=X_d$.
\end{lemma}

\begin{theorem}[Cioab\u{a} \cite{Cioa10}]
\label{ciothm}
Let $k$ and $d$ be two integers with $d\ge k\ge 2$. If $G$ is a
$d$-regular graph with $\lambda_2(G)< d-\frac{2(k-1)}{d+1}$, then
$\kappa'(G)\ge k$.
\end{theorem}
\begin{corollary}
\label{contrap}
Let $G$ be a $d$-regular graph with $d\ge 2$ and edge connectivity $\kappa'<d$.
Then $\lambda_2(G)\ge d-\frac{2\kappa'}{d+1}$.
\end{corollary}
\begin{proof}[\bf Proof:]
Let $k=\kappa'+1$ in the contrapositive of Theorem \ref{ciothm}.
\end{proof}

\section{Spectrum and generalized connectivity of regular graphs}
In this section, we prove Theorem \ref{L-conn}. Corollaries \ref{L-conn2} and \ref{2-conn} follow
from Theorem \ref{L-conn} obviously.

\begin{proof}[\bf Proof of Theorem~\ref{L-conn}:]
We prove it by contradiction and assume that $\kappa_l(G)< k$. By definition, there exists
a subset $S\subset V(G)$ with $|S|\le k-1$ such that $c(G-S)\ge l$. Let $s=|S|$, $c=c(G-S)$
and $H_1,H_2,\cdots,H_c$ be the components of $G-S$. For $1\le i\le c$, let $n_i=|V(H_i)|$
and $t_i$ be the number of edges between $H_i$ and $S$. Then $t_i\ge\kappa'$
for $1\le i\le c$. Since $G$ is $d$-regular, $\sum_{i=1}^c t_i \le ds \le d(k-1)$.

As $d(k-1)\ge \sum_{i=1}^c t_i\ge c\kappa'\ge l\kappa'$, we have $ld -d(k-1)\le ld -l\kappa'$.
If $\kappa'=d$, then the previous inequality is impossible, a contradiction. 
Thus we may assume that $\kappa' < d$, and hence $l\ge \frac{(l-k+1)d}{d-\kappa'}$.
We claim that there are at least $\lceil\frac{(l-k+1)d}{d-\kappa'}\rceil$ indices $i$
such that $t_i < d$. Otherwise, there are at most $\lceil\frac{(l-k+1)d}{d-\kappa'}\rceil -1$
indices $i$ such that $t_i < d$. In other words, there are at least
$c - \lceil\frac{(l-k+1)d}{d-\kappa'}\rceil +1$ indices $i$ with $t_i \ge  d$.
Thus 
\begin{eqnarray*}
\sum_{i=1}^c t_i 
&\ge & (c - \lceil\frac{(l-k+1)d}{d-\kappa'}\rceil +1)d 
       +(\lceil\frac{(l-k+1)d}{d-\kappa'}\rceil -1)\kappa' \\
&= & cd - (\lceil\frac{(l-k+1)d}{d-\kappa'}\rceil -1)(d-\kappa')\\
&> & cd - \frac{(l-k+1)d}{d-\kappa'}(d-\kappa')\\
&= & cd - (l-k+1)d =(c-l)d +(k-1)d \\
&\ge & ds, 
\end{eqnarray*}
contrary to $\sum_{i=1}^c t_i \le ds$. Hence there are at least 
$\lceil\frac{(l-k+1)d}{d-\kappa'}\rceil$ indices $i$ such that $t_i < d$. 
Without loss of generality, we may assume these indices are $1,2,\cdots, 
\lceil\frac{(l-k+1)d}{d-\kappa'}\rceil$.

For $1\le i\le \lceil\frac{(l-k+1)d}{d-\kappa'}\rceil$, $n_i\ge d+1$. Otherwise, 
if $n_i\le d$, then $dn_i=t_i + 2|E(H_i)|\le t_i+ n_i(n_i -1)\le t_i + d(n_i-1)$, 
which implies $t_i\ge d$, contrary to $t_i<d$.

Since $dn_i=t_i + 2|E(H_i)|$ for $1\le i\le \lceil\frac{(l-k+1)d}{d-\kappa'}\rceil$, 
if $d$ is even, then $t_i$ is also even, and thus $t_i\le d-2$. 
If $d$ is odd, then $t_i\le d-1$. As $n_i\ge d+1$, each $H_i$ contains at least 
three vertices of degree $d$ if $d$ is even and at least two vertices of degree $d$
if $d$ is odd. Thus $H_i\in \mathcal{X}_d$ for $1\le i\le \lceil\frac{(l-k+1)d}{d-\kappa'}\rceil$.
By Corollary \ref{lamp} and Lemma \ref{cwlem}, $\lambda_{\lceil\frac{(l-k+1)d}{d-\kappa'}\rceil} (G)
\ge \min_{1\le i\le \lceil\frac{(l-k+1)d}{d-\kappa'}\rceil}\{\lambda_1(H_i)\} \ge \theta(d)$, 
contrary to the assumption. This finishes the proof.
\end{proof}

\section{Spectrum and toughness of regular graphs}
In this section, we prove Theorems \ref{onetough}, \ref{bipar}, \ref{tautough} and Corollary \ref{brouconj}.

\begin{proof}[\bf Proof of Theorem~\ref{onetough}]
We prove it by contradiction and assume that $t(G)< 1$. By definition, there exists
a subset $S\subset V(G)$ such that $\frac{|S|}{c(G-S)}< 1$. Let $s=|S|$, $c=c(G-S)$
and $H_1,H_2,\cdots,H_c$ be the components of $G-S$. For $1\le i\le c$, let $n_i=|V(H_i)|$
and $t_i$ be the number of edges between $H_i$ and $S$. Then $s<c$ and $t_i\ge\kappa'$
for $1\le i\le c$. Since $G$ is $d$-regular, $\sum_{i=1}^c t_i \le ds$.

As $c\kappa'\le \sum_{i=1}^c t_i\le ds\le d(c-1)$, we have $c(d-\kappa')\ge d$. If $\kappa'=d$, then
we get a contradiction. Thus we may assume that $\kappa' < d$, and so $c\ge \frac{d}{d-\kappa'}$.
We claim that there are at least $\lceil\frac{d}{d-\kappa'}\rceil$ indices $i$
such that $t_i < d$. Otherwise, there are at most $\lceil\frac{d}{d-\kappa'}\rceil -1$
indices $i$ such that $t_i < d$. In other words, there are at least
$c - \lceil\frac{d}{d-\kappa'}\rceil +1$ indices $i$ with $t_i \ge  d$. Thus
\begin{eqnarray*} 
\sum_{i=1}^c t_i 
&\ge & (c - \lceil\frac{d}{d-\kappa'}\rceil +1)d + (\lceil\frac{d}{d-\kappa'}\rceil -1)\kappa'\\
& =  & cd - (\lceil\frac{d}{d-\kappa'}\rceil -1)(d-\kappa') \\
& >  & cd - \frac{d}{d-\kappa'}(d-\kappa') =  cd - d \\
&\ge & ds, 
\end{eqnarray*}
contrary to $\sum_{i=1}^c t_i \le ds$.
Thus there are at least $\lceil\frac{d}{d-\kappa'}\rceil$ indices $i$ such that $t_i < d$. Without loss
of generality, we may assume these indices are $1,2,\cdots, \lceil\frac{d}{d-\kappa'}\rceil$.

For $1\le i\le \lceil\frac{d}{d-\kappa'}\rceil$, $n_i\ge d+1$. Otherwise, if $n_i\le d$, then
$dn_i=t_i + 2|E(H_i)|\le t_i+ n_i(n_i -1)\le t_i + d(n_i-1)$, which implies $t_i\ge d$, contrary to $t_i<d$.

Since $dn_i=t_i + 2|E(H_i)|$ for $1\le i\le \lceil\frac{d}{d-\kappa'}\rceil$, if $d$ is even, then $t_i$ 
is also even, and thus $t_i\le d-2$. If $d$ is odd, then $t_i\le d-1$. As $n_i\ge d+1$, each $H_i$
contains at least three vertices of degree $d$ if $d$ is even and at least two vertices of degree $d$
if $d$ is odd. Thus $H_i\in \mathcal{X}_d$ for $1\le i\le \lceil\frac{d}{d-\kappa'}\rceil$.
By Corollary \ref{lamp} and Lemma \ref{cwlem}, $\lambda_{\lceil\frac{d}{d-\kappa'}\rceil} (G)
\ge \min_{1\le i\le \lceil\frac{d}{d-\kappa'}\rceil}\{\lambda_1(H_i)\} \ge \theta(d)$, 
contrary to the assumption. This finishes the proof.
\end{proof}

\begin{lemma}
\label{biparlem}
For any bipartite regular graph $G$, $t(G)\le 1$.
\end{lemma}
\begin{proof}[\bf Proof:]
Let $S$ be the set of vertices of one part of the bipartition. Then $c(G-S)=|S|$.
Thus $t(G)\le \frac{|S|}{c(G-S)}=1$.
\end{proof}

\begin{proof}[\bf Proof of Theorem~\ref{bipar}:]
We prove it by contradiction and assume that $t(G)\neq 1$. By Lemma \ref{biparlem}, $t(G)<1$.
By definition, there exists a subset $S\subset V(G)$ such that $\frac{|S|}{c(G-S)}< 1$.
Similar argument as Theorem~\ref{onetough} shows that there are at least 
$\lceil\frac{d}{d-\kappa'}\rceil$ components $H_i$ of $G-S$ such that $t_i < d$, where
$t_i$ is the number of edges between $H_i$ and $S$ for $1,2,\cdots, \lceil\frac{d}{d-\kappa'}\rceil$.
Let $n_i=|V(H_i)|$ and $m_i=|E(H_i)|$ for $1,2,\cdots, \lceil\frac{d}{d-\kappa'}\rceil$. 
Then $2m_i=dn_i - t_i \ge dn_i -d+1$. As each $H_i$ is also bipartite, $m_i\le n_i^2/4$.
Thus $n_i^2/2 \ge 2m_i \ge dn_i -d+1$, which implies that $n_i^2 -2dn_i +2d -2\ge 0$.
Hence $n_i\ge 2d$. By Corollary \ref{lamp}, 
\[
\lambda_{\lceil\frac{d}{d-\kappa'}\rceil} (G)
\ge \min_{1\le i\le \lceil\frac{d}{d-\kappa'}\rceil}\{\lambda_1(H_i) \} 
\ge \min_{1\le i\le \lceil\frac{d}{d-\kappa'}\rceil}\{\frac{2m_i}{n_i} \}
\ge \frac{dn_i -d+1}{n_i} \ge d-\frac{d-1}{2d},
\]
contrary to the assumption. This finishes the proof.
\end{proof}

\begin{proof}[\bf Proof of Theorem \ref{tautough}:]
Suppose that $S$ is a vertex-cut of $G$. Let $s=|S|$, $c=c(G-S)$
and $H_1,H_2,\cdots,H_c$ be the components of $G-S$. For $1\le i\le c$, let $n_i=|V(H_i)|$
and $t_i$ be the number of edges between $H_i$ and $S$. Then $t_i\ge\kappa'$ for $1\le i\le c$.
As $G$ is $d$-regular, $\sum_{i=1}^c t_i \le ds$.
Thus $c\kappa'\le \sum_{i=1}^c t_i \le ds$, which implies that $s/c \ge \kappa'/d$.
Hence $t(G)\ge \kappa'/d$.
\end{proof}

\begin{proof}[\bf Proof of Corollary \ref{brouconj}:]
By Corollary \ref{contrap}, $\lambda_2\ge d-\frac{2\kappa'}{d+1}$, which implies that
$\frac{2\kappa'}{\lambda_2(d+1)}\ge \frac{d}{\lambda_2}-1$.
If $d\ge 4$, then $\lambda_2\ge d-\frac{2\kappa'}{d+1}>2$. If $d=3$, then $\kappa'\le 2$,
and thus $\lambda_2\ge d-\frac{2\kappa'}{d+1}\ge 2$. By Theorem \ref{tautough},
\[
t(G)\ge \kappa'/d > \frac{\kappa'/d}{\frac{\lambda_2}{2}(1+\frac{1}{d})}
=\frac{2\kappa'}{\lambda_2(d+1)}\ge \frac{d}{\lambda_2}-1,
\]
which completes the proof.
\end{proof}

\section{Spectrum and spanning $k$-trees in regular graphs}

In this section, we prove Theorem \ref{kspanning}. We will use the following sufficient condition of 
the existence of a spanning $k$-tree obtained by Win~\cite{Win89}, which was also proved by 
Ellingham and Zha~\cite{ElZh00} with a new proof later.

\begin{theorem}[Ellingham and Zha~\cite{ElZh00}, Win~\cite{Win89}]
\label{EZWthm}
Let $k\ge 2$ and $G$ be a connected graph. If for any $S\subseteq V(G)$, $c(G-S)\le (k-2)|S| +2$,
then $G$ has a spanning $k$-tree.
\end{theorem}

Now we are ready to prove Theorem \ref{kspanning}.

\begin{proof}[\bf Proof of Theorem \ref{kspanning}]
We prove it by contradiction and assume that $G$ does not have spanning $k$-trees for $k\ge 3$.
By Theorem~\ref{EZWthm}, there exists a subset $S\subseteq V(G)$
such that 
\begin{equation}
\label{csrel}
c(G-S)\ge (k-2)|S| +3. 
\end{equation}
Let $s=|S|$, $c=c(G-S)$
and $H_1,H_2,\cdots,H_c$ be the components of $G-S$. For $1\le i\le c$, let $n_i=|V(H_i)|$
and $t_i$ be the number of edges between $H_i$ and $S$. Then $t_i\ge\kappa'$
for $1\le i\le c$. Since $G$ is $d$-regular, $c\kappa'\le \sum_{i=1}^c t_i \le ds$.
By (\ref{csrel}), $s\le (c-3)/(k-2)$. Thus $c\kappa'\le d(c-3)/(k-2)$, which implies that
\begin{equation}
\label{cform}
c\left(d-(k-2)\kappa'\right)\ge 3d.
\end{equation}
Thus $l= d-(k-2)\kappa' > 0$, contrary to (i). This proves (i). In the following,
we continue to prove (ii).

By (\ref{cform}), $c\ge \lceil\frac{3d}{l}\rceil$. 
We claim that there are at least $\lceil\frac{3d}{l}\rceil$ indices $i$ such that 
$t_i < d/(k-2)$. Otherwise, there are at most $\lceil\frac{3d}{l}\rceil -1$
indices $i$ such that $t_i < d/(k-2)$. In other words, there are at least
$c - \lceil\frac{3d}{l}\rceil +1$ indices $i$ with $t_i \ge d/(k-2)$. Thus
\begin{eqnarray*} 
ds\ge \sum_{i=1}^c t_i 
&\ge & (c - \lceil\frac{3d}{l}\rceil +1)\cdot\frac{d}{k-2} + (\lceil\frac{3d}{l}\rceil -1)\kappa'\\
& =  & \frac{cd}{k-2} - (\lceil\frac{3d}{l}\rceil -1)(\frac{d}{k-2}-\kappa') \\
& >  & \frac{cd}{k-2} - \frac{3d}{l}\cdot (\frac{d}{k-2}-\kappa')\\
& =  & \frac{cd}{k-2} - \frac{3d}{k-2} =d\cdot \frac{c-3}{k-2}\\
&\ge & ds, 
\end{eqnarray*}
a contradiction. This proves that there are at least 
$\lceil\frac{3d}{l}\rceil$ indices $i$ such that $t_i < d/(k-2)$. Without loss
of generality, we may assume these indices are $1,2,\cdots, \lceil\frac{3d}{l}\rceil$.

For $1\le i\le \lceil\frac{3d}{l}\rceil$, since $t_i < d/(k-2)$, it is not hard to get $n_i\ge d+1$
by counting total degree of $H_i$. By Corollary \ref{lamp}, $\lambda_{\lceil\frac{3d}{l}\rceil} (G)
\ge \min_{1\le i\le \lceil\frac{3d}{l}\rceil}\{\lambda_1(H_i)\} \ge d-\frac{d}{(k-2)(d+1)}$, 
contrary to the assumption. This finishes the proof.
\end{proof}

\section{Final Remarks}

In this paper, we determined some new connections between the spectrum of a regular graph and its generalized connectivity, toughness or the existence of spanning $k$-trees. Some of our results are best possible. For example, the constructions from \cite[Section 3]{CiWo14} show that the upper bound from Theorem \ref{onetough} is best possible. Also, Corollary \ref{2-conn} is best possible when $d=4$. To see this, construct a $4$-regular graph by taking two disjoint copies of $X_4$ and adding a new vertex adjacent to the $4$ vertices ($2$ in each copy of $X_4$) of degree $3$. The resulting graph is $4$-regular, has vertex-connectivity $1$ and its second largest eigenvalue equals the upper bound from Corollary \ref{2-conn}.

It would interesting to improve and generalize our results to general graphs and eigenvalues of Laplacian matrix, signless Laplacian or normalized Laplacian.

\end{document}